\documentclass{article}
\usepackage{amsmath,amsfonts,amssymb,amsthm}
 \usepackage{color}
\usepackage{fancybox}

\usepackage{hyperref}
\AtBeginDocument{
  \label{CorrectFirstPageLabel}
  
}

\newcommand{\I}{{\bf 1}}

\theoremstyle{plain}
\newtheorem{theorem}{Theorem}[section]

\newtheorem{corollary}[theorem]{Corollary}
\newtheorem{proposition}[theorem]{Proposition}

\theoremstyle{definition}

\theoremstyle{remark}
\newtheorem{remark}{Remark}

\newcommand{\nc}{\newcommand}
\nc{\R}{{\mathbb R}}
\nc{\N}{{\mathbb N}}
\renewcommand{\S}{{\mathbb S}}
\nc{\Z}{{\mathbb Z}}
\nc{\BP}{\mathbb{P}}
\renewcommand{\P}{\mathbb{P}}
\DeclareMathOperator{\BE}{\mathbb{E}}
\nc{\BQ}{\mathbb{Q}}
\nc{\bD}{{\mathbf D}}
\nc{\bF}{{\mathbf F}}
\nc{\bN}{{\mathbf N}}
\nc{\BX}{{\mathbb X}}
\nc{\BS}{{\mathbb S}}
\nc{\NB}{{\mathsf{NB}}}
\nc{\Bin}{{\mathsf{Bin}}}
\nc{\Po}{{\mathsf{Po}}}
\nc{\CPo}{{\mathsf{CPo}}}
\nc{\Er}{{\mathsf{Er}}}
\nc{\cB}{{\mathcal B}}
\nc{\deq}{\overset{D}{=}}
\nc{\comp}[1]{{#1}^{\mathbf{c}}}
\nc{\stas}{St$\alpha$S}
\nc{\thru}{,\dotsc,}
\renewcommand{\epsilon}{\varepsilon}
\renewcommand{\phi}{\varphi}

\DeclareMathOperator{\card}{card}

\DeclareMathOperator{\cone}{cone}

\DeclareMathOperator{\Int}{Int}
\DeclareMathOperator{\exo}{exo}

\nc{\seg}{{see, e.g.\ }}

\begin{document}

\author{G\"unter Last\thanks{Karlsruhe Institute of Technology,
   Institut f\"ur Stochastik, Kaiserstra{\ss}e 89, D-76128 Karlsruhe,
   Germany. Email: \texttt{guenter.last@kit.edu}} and
Sergei Zuyev \thanks{Chalmers University of Technology and
   University of Gothenburg, 
Gothenburg, Sweden. E-mail: \texttt{sergei.zuyev@chalmers.se}}}

\title{Applications of the perturbation formula for Poisson processes to elementary and geometric probability}
\date{\today}
\maketitle

\begin{abstract}
We present a unified approach to deriving integral representations for
the binomial, negative binomial, Poisson, compound Poisson, and Erlang
distributions with respect to their continuous parameters. This is
achieved using Margulis-Russo-type formulas for Bernoulli and Poisson
processes, which also provide a natural probabilistic interpretation
of their derivatives. Extending these variational methods, we derive
new integro-differential identities that characterise the densities of
strictly $\alpha$-stable multivariate distributions.
We further generalize Crofton's derivative formula from integral
geometry to the case of Poisson processes. This extension allows us
to establish a new probabilistic proof of the formula for binomial
point processes, highlighting the underlying geometric structure in a
probabilistic framework.
\end{abstract}



\section{Introduction}\label{secintro}
The purpose of this article is to demonstrate how fundamental
derivative formulas, based on the principle of pivotality, can be used
to derive a wide range of distributional identities. Some of these
identities, particularly for classical univariate distributions, are
well-known and can be obtained by standard methods or direct
computation. Others, involving multivariate strictly stable
distributions and Crofton's formulas, appear here in a generality that,
to the best of our knowledge, has not been previously established.
Our method not only shows that all these formulas follow from the same
principle but also provides a clear probabilistic interpretation: the
derivatives represent the expected number, or, more generally,
the measure, of certain \emph{pivotal elements}. This interpretation gives
new insight into the structure of these identities.
The key elements of our approach are two main formulas: one for
Bernoulli systems and one for Poisson point processes. These serve as
the basis for deriving and understanding the distributional identities
presented below.

Consider a vector $X=(X_1,\dots,X_m), m\in\N,$ of independent
Bernoulli random variables, defined on its canonical probability space
$\Omega=\{0,1\}^m$ with $X_i$ having the success probability
$p_i\in[0,1],\ i=1\thru m$.  For an $i\in\{1,\dots,m\}$, let $X_{(i)}$ (resp.\
$X^{(i)}$) be a vector whose entries coincide with those of $X$ except
at the $i$-th coordinate, where the entry is $0$ (resp.\
$1$). A coordinate $i$ is called \emph{pivotal} for an event
$A\subset \Omega$ in configuration $X$, if $\I_A(X^{(i)})\neq
\I_A(X_{(i)})$. Note that
\begin{displaymath}
  \I_A(X)=X_i\I_A(X^{(i)}) + (1-X_i) \I_A(X_{(i)}),
\end{displaymath}
and because $X_i$ is independent of $\I_A(X^{(i)})$ and of $\I_A(X_{(i)})$,
\begin{align}
\BP(A)&=\BE \I_A(X)=p_i\BE\I_A(X^{(i)}) + (1-p_i) \BE\I_A(X_{(i)}), \notag  \\
\frac{\partial}{\partial p_i}\BP(A) &= \BE[\I_A(X^{(i)}) -\I_A(X_{(i)})]. \label{eq:ep} 
\end{align}
The term in the right side of \eqref{eq:ep} is known as the
\emph{influence} of the coordinate $i$ on the event $A$,  see, e.g.,
\cite[p.\,46]{BoRio06}.

Introduce an operator $D_i$ acting on a function
$g:\Omega\to\R$ as
\begin{displaymath}
  D_ig(x) := g(x^{(i)}) - g(x_{(i)}),\ x\in\Omega.
\end{displaymath}
Consider the case when $p_i=p$ for all $i$ and add an
index $p$ to the probability and expectation notation to show its value.
Now, \eqref{eq:ep} yields
\begin{equation}
  \label{mr}
  \frac{d}{dp}\BP_p(A) = \sum_{i=1}^m \frac{\partial}{\partial
    p_i}\Bigr|_{p_i=p}\BP(A) = \BE_p\sum_{i=1}^m D_i\I_A(X).
\end{equation}
Because only the terms corresponding to pivotal coordinates are non-zero in the
above sum, the derivative of the probability $\BP_p(A)$ equals the
expected number of pivotal coordinates for the event $A$. This is the
so-called \emph{Margulis--Russo formula} 
which was first proved in \cite{EsaPro:63}, where the authors call pivotal
coordinates \emph{essential}.
The formula was  independently
rediscovered in \cite{Marg:74} and \cite{Russo:81}, with a focus on
increasing events, i.e.\ such events where $\I_A$ is non-decreasing:
$D_i\I_A(x)\geq 0$ for all all $i$ and $x\in\Omega$.

Since any function on $\Omega$ is a linear combination of indicator functions,
\eqref{mr} extends to a {\em perturbation formula} for a general function:
\begin{equation}\label{mrrv}
  \frac{d}{dp} \BE_p g(X)=\BE_p \sum^m_{i=1} D_i g(X).
\end{equation}
An analogue of $D_i$ is the difference operator $D_z$ defined in
\eqref{eq:dz} as it appears in the perturbation formula for
Poisson processes in Section~\ref{secPoisson}.

In practice, it is often more convenient to use a modified form of the
Margulis--Russo formula proposed in
\cite{Men:87}. Assuming that $p\in(0,1)$ and using again independence
of $X_i$ from $D_i\I_A$, continue \eqref{mr} as follows:
\begin{align*}
  &\BE_p\sum_{i=1}^m D_i\I_A(X)\I\{D_i\I_A(X)\neq 0\}
  = \frac{1}{p} \BE_p\sum_{i=1}^m \I_A(X^{(i)})\I\{D_i\I_A(X)\neq
    0\}\I\{X_i=1\} \\
  & -\frac{1}{1-p} \BE_p\sum_{i=1}^m \I_A(X_{(i)})\I\{D_i\I_A(X)\neq
  0\}\I\{X_i=0\}.
\end{align*}
Introducing a \emph{flip operator} $T_i$ acting on $x\in\Omega$ as
\begin{displaymath}
  T_i (x_1\thru x_i\thru x_m)= (x_1\thru 1-x_i\thru x_m),
\end{displaymath}
the sums above can be written as
\begin{align}
  \label{pivot}
  N^+_A(X)=\sum_{i=1}^m \I\{X_i=1, X\in A, T_iX\not\in A\},\quad
 N^-_A(X)=\sum_{i=1}^m \I\{X_i=0, X\in A, T_iX\not\in A\}.
\end{align}
The coordinates $i$ which contribute non-zero terms to $N_A^+$ are
called $(+)$\emph{-pivotal} for event $A$ in configuration $X$
(respectively, $(-)$\emph{-pivotal} terms in $N_A^-$). We thus obtain the
following form of the Margulis--Russo formula:
\begin{equation}
  \label{marg2}
  \frac{d}{dp}\BP_p(A) = \frac{1}{p}\BE_p N_A^{+} - \frac{1}{1-p}\BE_p
  N_A^{-},\quad p\in(0,1).
\end{equation}
In the case of an increasing $A$, there are no (-)-pivotal coordinates
so the last term above is zero. 
Formula~\eqref{marg2} was used in
\cite{Zue:92b} to establish an analogous perturbation formula for
Poisson point processes which are covered in the next section. It is also
easier to apply because it involves counting only
pivotal coordinates for $X\in A$ ignoring the ones for $X\not\in A$
involved in $D_i\I_A$.

The power of the Margulis--Russo formula lies in the fact that it links the
variation of an event's probability to the geometry of configurations
comprising the event, i.e.\ the number of pivotal components for its
occurence. Many results in the percolation theory exploit this link,
see, e.g., \cite{Grim99} or \cite{BoRio06}.

To give a flavour, consider the binomial distribution
\begin{displaymath}
\Bin(n,p;k):=\binom{n}{k}p^k(1-p)^{n-k},\quad k\in\{0,\dots,n\}
\end{displaymath}
with parameters $n\in\N=\{1,2,\dots\}$ and $p\in[0,1]$.
Take $k\in\{1,\dots,n\}$, and consider the increasing event
$A:=\{S_n\ge k\}$, where $S_n:=X_1+\dots +X_n$. 
We then have
\begin{align*}
N^+_A=\begin{cases}
k,&\text{if $S_n=k$},\\
0,&\text{otherwise}.
\end{cases}
\end{align*}
Invoking \eqref{marg2}, leads to
\begin{displaymath}
\frac{d}{dp}\BP_p (A) = \frac{k}{p} \Bin(n,p;k)=\frac{n!}{(k-1)!(n-k)!}p^{k-1}(1-p)^{n-k}.
\end{displaymath}
Since $\BP_0(A)=0$, we obtain the following integral representation:
\begin{align}\label{bin}
\sum_{j=k}^n \Bin(n,p;j)=\frac{n!}{(k-1)!(n-k)!}\int^p_{0}t^{k-1}(1-t)^{n-k}dt,
\quad k\in\{1,\dots,n\}.
\end{align}

Another example concerns the negative binomial (aka Pascal) distribution
\begin{displaymath}
\NB(r,p;k):=\binom{k+r-1}{k}p^r(1-p)^k,\quad k\in\Z_+,
\end{displaymath}
with parameters $r\in\N$ and $p\in[0,1]$. This is the distribution of
the number of failures in a sequence of independent Bernoulli trials
until the $r$-th success.  Fix natural numbers $r, k$. We can take $m:=k+r-1$
to see that
\begin{align*}
\sum_{j=0}^k \NB(r,p;j)=\BP(A),
\end{align*}
where $A$ denotes the increasing event $\{S_{k+r-1}\ge r\}$. We then have
\begin{align*}
N^+_A=\begin{cases}
r,&\text{if $S_{k+r-1}=r$},\\
0,&\text{otherwise},
\end{cases}
\end{align*}
and \eqref{marg2} yields that
\begin{displaymath}
\frac{d}{dp}\BP_p (A) =\frac{r}{p}\Bin(k+r-1,p;r)
=\frac{(k+r-1)!}{(k-1)!(r-1)!}p^{r-1}(1-p)^{k-1}.
\end{displaymath}
Since $\BP_0(A)=0$, we get the identity
\begin{align}\label{negbin}
\sum_{j=0}^k \NB(r,p;j)=\frac{(k+r-1)!}{(k-1)!(r-1)!}\int^p_{0}t^{r-1}(1-t)^{k-1}dt,
\quad k\in\N.
\end{align}
Both \eqref{bin} and \eqref{negbin} could be checked
directly by differentiation, but the variation formula \eqref{mr} provides a
probabilistic insight into the variation of the probability.

The aim of this paper is to demonstrate the power of an analogous to
\eqref{mr} variation formula for Poisson processes. Let $\lambda$ be a
finite measure on some measurable space $\BX$ and $\theta\ge
0$. Consider a point process $\eta$ on $\BX$ and a probability measure
$\BP_\theta$ such that $\eta$ is (under $\BP_\theta$) a Poisson
process with intensity measure $\theta \lambda$. It was shown in
\cite{Zue:92b} that, if $A$ is an event defined in terms of $\eta$,
then \eqref{mr} holds with $\theta$ replacing $p$, provided that
\eqref{pivot} is modified as follows:
\begin{align}\label{pivotpoisson}
N^+_A:=\int\I\{\eta+\delta_z \in A,\eta\notin A\}\lambda(dz), 
\qquad N^-_A:=\int\I\{\eta\in A,\eta+\delta_z\notin A\}\lambda(dz),
\end{align}
where, generically, $\delta_z$ is the Dirac measure at
$z$. Notice, that by Mecke's formula (\seg \cite[Th.~4.1]{LastPenrose18}),
\begin{displaymath}
  \BE_\theta N_A^+=\frac{1}{\theta}\, \BE_\theta
  \int \I\{\eta \in A,\eta-\delta_z\notin A\}\eta(dz),
\end{displaymath}
So, analogously to the Bernoulli case, the process points $z\in\eta$
such that $\eta \in A$, but $\eta-\delta_{z}\notin A$ maybe called
\emph{pivotal points}, whereas $z \in \BX$ such that $\eta\in A$, but
$\eta+\delta_z\notin A$ are called \emph{pivotal locations}.

In the next section we provide a review of the variation formulas for
Poisson processes and to prove with their help distributional
identities for some classical one-dimensional distributions: Poisson,
Erlang and compound Poisson. We then consider multivariate strictly
stable distributions in
Section~\ref{sec:strictly-alpha-stabl}. Theorem~\ref{th:stab} provides
formulas which are apparently new and could possibly be used for an
effective computation of the stable density which is not known
explicitly apart for a very few particular values of its parameters.
In Section~\ref{secCroftonPoiss}, we extend Crofton's derivative
formula. In the final Section \ref{secBinomial} we use this extension
to give a new probabilistic proof of a version of this formula for
binomial point processes.

\section{A perturbation formula for Poisson processes}\label{secPoisson}

In this section we review a perturbation formula for general Poisson
processes.  Let $(\BX,\mathcal{X})$ be a Borel space and let
$\bN(\BX)\equiv \bN$ be the space of integer-valued $\sigma$-finite
measures $\varphi$ on $\BX$, equipped with the smallest $\sigma$-field
$\mathcal{N}$ making the mappings $\varphi\mapsto\varphi(B)$
measurable for all $B\in\mathcal{X}$.


For any $g\colon\bN\rightarrow\R$ and $z\in\BX$,
introduce a function $D_zg\colon \bN\to\R$ by means
of
\begin{equation}
  D_zg(\phi)=g(\phi+\delta_z)-g(\phi).\label{eq:dz}
\end{equation}
The mapping $g\mapsto D_zg$ is known as {\em difference operator}.
For $k\in\N$ the $k$-th iteration $D^{k}_{z_1,\ldots,z_k}g:\bN\rightarrow\R$,
of this operator is inductively defined by 
$D^{k}_{z_1,\ldots,z_k}g=D_{z_k} D^{k-1}_{z_1,\ldots,z_{k-1}}g$ for
$(z_1,\ldots,z_k)\in\BX^k$. 


Given any $\sigma$-finite measure
$\rho$ on $\BX$, we let $\eta_\rho$ denote a Poisson process with
this intensity measure. The following perturbation formula is
a special case of Theorem 19.3 in \cite{LastPenrose18}.

\begin{theorem}\label{tanalytic} Let $\lambda$  be a $\sigma$-finite 
and let $\nu$ be a finite measure on $\BX$.
Let $g\colon\bN\rightarrow\R$ be a measurable
function such that $\BE |g(\eta_{\lambda+\nu})|<\infty$.
Let $\theta\in(-\infty,1]$ such that $\lambda+\theta\nu$ is a measure. Then
\begin{align}\label{emar4}
\BE g(\eta_{\lambda+\theta\nu})=\BE g(\eta_\lambda)+
\sum^\infty_{k=1}\frac{\theta^k}{k!}\int \BE D^k_{x_1,\ldots,x_k}g(\eta_\lambda)\,
\nu^k(d(x_1,\ldots,x_k)),
\end{align}
where the series converges absolutely.
\end{theorem}

The earliest version of Theorem \ref{tanalytic} (for a bounded function $g$) was
proved in \cite{Zue:92b}. Later this was generalised in \cite{MolZu00}. 
For square integrable random variables the result can be extended to 
certain (signed) $\sigma$-finite perturbations; see \cite{Last14}.

For later reference we provide the following consequence of
Theorem \ref{tanalytic} (set $\nu=\lambda$ in \eqref{emar4}).

\begin{proposition}\label{p2} Let $\lambda$ be a finite measure on $\BX$.
Let $g\colon \bN\rightarrow\R$ be a measurable function such
that $\BE |g(\eta_{\theta_0\lambda})|<\infty$ for some $\theta_0>0$. 
Then $\theta\mapsto\BE g(\eta_{\theta\lambda})$
is analytic on $[0,\theta_0]$ and its derivatives are given by
\begin{align}\label{russopoiss}
 \frac{d^k}{d\theta^k}\BE g(\eta_{\theta\lambda})
=\int\cdots\int \BE D^{k}_{z_1,\ldots,z_k}g(\eta_{\theta\lambda})
  \,\lambda(dz_1)\cdots\lambda(dz_k), 
\quad \theta\le\theta_0.  
\end{align}
\end{proposition}

Using the indicator $\I_A$ above as the function $g$,
implies the Poisson process version of the Margulis-Russo formula
\eqref{mr} with notation \eqref{pivotpoisson}.

\paragraph*{The Poisson and the Erlang distribution.}

We now apply \eqref{russopoiss} to derive a few distributional
identities for classical univariate distributions.
The {\em Poisson distribution} with parameter $\theta\ge 0$
is given by
\begin{displaymath}
\Po(\theta;k):=\frac{\theta^k}{k!}e^{-\theta},\quad k \in \Z_+.
\end{displaymath}
The {\em Erlang distribution} with parameters $n\in\N$ and $\theta>0$
has density function
\begin{displaymath}
\Er(n,\theta;x):=\frac{\theta^n}{(n-1)!}x^{n-1} e^{-\theta x},\quad x\ge 0.
\end{displaymath}

\begin{proposition}
The Poisson distribution with parameter $\theta\ge 0$ satisfies
\begin{align}\label{Poisson}
\sum_{j=k}^\infty \Po(\theta;j)=\int^\theta_{0}\frac{t^{k-1}}{(k-1)!}e^{-t}dt,
\quad k\in\N.
\end{align}
The distribution function of the Erlang distribution with parameters $n\in\N$ and $\theta>0$ 
may be written as
\begin{align}\label{Erlang}
\int_0^x\Er(n,\theta;y)dy=\frac{x^n}{(n-1)!} \int_0^\theta t^{n-1} e^{-t x}dt,\quad x\ge 0.
\end{align}
\end{proposition}
\begin{proof}
  In the notation of the previous section, set $\BX=\{z\}$ to be a
  singleton and $\lambda\{z\}=1$. Then 
  $\eta\{z\}$ is just a Poisson random variable with parameter 1. Take
  $k\in\N$ and consider the event $A:=\{\eta\{z\}\ge k\}$. Then
  $\BP_\theta(A)$ is given by the left-hand side of \eqref{Poisson}
  and we have
  \begin{displaymath}
    \I_A(\eta+\delta_z)-\I_A(\eta)=\I\{\eta\{z\}= k-1\}.
  \end{displaymath}
  Since $\BP_0(A)=0$, \eqref{russopoiss} implies \eqref{Poisson}.
  Eq.~\eqref{Erlang} follows from \eqref{Poisson} and the identity
  \begin{align}\label{Erlang2}
    \int_0^x\Er(n,\theta;y)dy=\sum^\infty_{j=n}\Po(\theta x;j).
  \end{align}

\end{proof}

\paragraph*{The Compound Poisson distribution.}
\label{sec:comp-poiss-distr}

Let $\BQ$ be a probability distribution on $\R$. The
{\em compound Poisson distribution} with parameters $\theta\ge 0$
and $\BQ$ is given by
\begin{displaymath}
\CPo(\theta,\BQ):=\sum^\infty_{n=0} \frac{\theta^n}{n!}e^{-\theta}\BQ^{\ast n},
\end{displaymath}
so it equals the distribution of Poisson $\Po(\theta)$ number of
independent summands each having distribution $\BQ$.

\begin{proposition}
  \label{prop:cpois}
  The distribution function $F(\theta,\BQ;x):=\CPo(\theta,\BQ)((-\infty,x])$,
  $x\in\R$, of the Compound Poisson distribution satisfies
  \begin{align}\label{CPoisson3}
\frac{d}{d\theta}F(\theta,\BQ;x)
=\int F(\theta,\BQ;x-z)\BQ(dz) -F(\theta,\BQ;x).
\end{align}
When $\BQ$ is concentrated on $\Z$, 
\begin{align}\label{CPoisson1}
\frac{d}{d\theta}\CPo(\theta,\BQ;k)
=\sum_{j\ne k}q_{k-j}\CPo(\theta,\BQ;j)
-(1-q_0)\CPo(\theta,\BQ;k),\quad k\in\Z,
\end{align}
where 
$\CPo(\theta,\BQ;j):=\CPo(\theta,\BQ)(\{j\})$, and
$q_j:=\BQ(\{j\})$, $j\in\Z$. Equivalently,
\begin{align}\label{CPoisson2}
\frac{d}{d\theta}\CPo(\theta,\BQ;k)
=\sum_{j\in\Z}q_{k-j}\CPo(\theta,\BQ;j) -\CPo(\theta,\BQ;k),\quad k\in\Z.
\end{align}
\end{proposition}
\begin{proof}
  
  To apply \eqref{russopoiss}, take $\BX:=\R$ and let $\lambda:=\BQ$.
  Under $\BP_\theta$, the random variable $Z:=\int z\eta(dz)$ has the
  compound Poisson distribution $\CPo(\theta,\BQ)$.  Consider the
  event $A:=\{Z\le x\}$ for some $x\in\R$.  Then, for $z\in\R$,
  \begin{displaymath}
\I_A(\eta+\delta_z)-\I_A(\eta)=\I\{Z>x,Z+z\le x\}-\I\{Z\le x,Z+z>x\},
\end{displaymath}
so that \eqref{russopoiss} writes
\begin{align*}
\frac{d}{d\theta}\BP_\theta(A)=&
\BE_\theta\int_{(-\infty,0)}\I\{Z+z\le x\}(1-\I\{Z\le x\})\BQ(dz)\\
&-\BE_\theta\int_{(0,\infty)}(1-\I\{Z+z\le x\})\I\{Z\le x\}\BQ(dz)\\
=&
\BE_\theta\int_{\R\setminus\{0\}}\I\{Z+z\le x\}\BQ(dz)-
\BP_\theta(Z\le x)\BQ(\R\setminus\{0\}).
\end{align*}
This yields \eqref{CPoisson3}. 
If $\BQ(\Z)=1$, then \eqref{CPoisson2} (and hence also 
\eqref{CPoisson1}) follows upon taking 
suitable differences. 
\end{proof}

\begin{remark}
  Identity \eqref{CPoisson1} is equivalent to
\begin{align}\label{rec}
\CPo(\theta,\BQ;k)
=e^{-(1-q_0)\theta}\sum_{j\ne k}q_{k-j}\int^\theta_0e^{(1-q_0)t}\CPo(t,\BQ;j)dt,
\quad k\in\Z.
\end{align}
If, in addition, $q_j=0$ for $j<0$, then
it follows from the definition of $\CPo(\theta,\BQ)$ (or from
$\CPo(\theta,\BQ;0)=e^{-(1-q_0)\theta}$ and \eqref{rec}) that
$e^{(1-q_0)\theta}\CPo(\theta,\BQ;k)$ is a polynomial in $\theta$
of degree $k$.  Equations \eqref{rec} provide a recursion
for the coefficients of these polynomials.
\end{remark}
\begin{remark}\label{rpanjer}
The characteristic function of $\CPo(\theta,\BQ;k)$ is given by
 \begin{align}\label{12}
 G(\theta,\BQ;s):=\int e^{\mathbf{i}sz}\CPo(\theta,\BQ)(dz)
 =\exp[\theta(G_{\BQ}(s)-1)],\quad s\in\R,
 \end{align}
 where $\mathbf{i}$ is the imaginary unit and
 $G_{\BQ}$ is the characteristic function of $\BQ$.
The recursion \eqref{rec} can also be obtained by differentiating \eqref{12}
with respect to $\theta$.
Differentiation of \eqref{12} with respect to $s$ 
and assuming $q_j=0$ for $j<0$,
yields the widely used Panjer recursion \cite{Panjer81}:
\begin{align}\label{recPanjer}
\CPo(\theta,\BQ;k)
=\theta\sum^{k-1}_{j=0}\frac{k-j}{k}q_{k-j} \CPo(\theta,\BQ;j),
\quad k\in\N.
\end{align}
\end{remark}

\begin{remark}\label{rcPoisson}
Consider a compound Poisson process $(X_t)_{t\ge 0}$
driven by a unit rate Poisson process and
jump size distribution $\BQ$, see e.g.\ \cite[Ch.~12]{Kallenberg02}.
Then $X_t$ has distribution $\CPo(t,\BQ)$ and
\eqref{CPoisson2} and \eqref{CPoisson3} are two examples for
the Kolmogorov forward equation, see e.g.\ 
\cite[Ch.~19]{Kallenberg02}.
\end{remark}

\begin{remark}
  Take $\BX:=[0,\infty)$ and $\lambda$ as the Lebesgue
  measure on $\BX$. Let $T_1<T_2<\dots$ be the atoms of $\eta$
  arranged in increasing order.  Let $n\in\N$ and consider the event
  $A:=\{T_n\le x\}$ for $x\ge 0$. It is well-known that
  $\BP_\theta(A)$ coincides with the left-hand side of \eqref{Erlang}.
  On the other hand we have for all $z\ge 0$ (with obvious notation)
  \begin{displaymath}
\I\{T_n(\eta+\delta_z)\le x\}-\I\{T_n(\eta)\le x\}=\I\{z\le
x\}\I\{\eta[0,x]=n-1\},
\end{displaymath}
so that \eqref{Poisson} yields
\begin{align*}
  \frac{d}{d\theta}\BP_{\theta}(A)
  &=x\BP\bigl(\eta[0,x]=n-1\bigr)=\frac{x^n}{(n-1)!}\theta^{n-1}e^{-\theta x}.
\end{align*}
Since $\BP_{0}(A)=0$, we again obtain \eqref{Erlang}.
\end{remark}

\section{Strictly $\alpha$-stable laws}
\label{sec:strictly-alpha-stabl}

A random vector $\xi$ (or its distribution) is called strictly
$\alpha$-stable (\stas), if the following equality in distribution holds:
\begin{equation}
  \label{eq:sas}
  t^{1/\alpha}\xi'+(1-t)^{1/\alpha}\xi'' \deq \xi,\quad
  0\leq t\leq 1,
\end{equation}
where $\xi', \xi''$ are independent distributional copies of $\xi$. In
Euclidean spaces, \stas\ laws exist only for $0<\alpha\leq2$ and
$\alpha=2$ corresponds to the Gaussian distribution centred at the
origin. Symmetrical \stas\ random vectors in $\R^n$ with $\alpha<2$ and all
\stas\ random vectors with $\alpha<1$ admit the following
\emph{LePage series representation} (see \cite{lp81}):
\begin{equation}
  \label{eq:lp}
  \xi\deq \sum_{k=1}^\infty \Gamma_k^{-1/\alpha}\epsilon_k,
\end{equation}
where $\Gamma_1,\Gamma_2,\dots$ are the successive times of jumps of a
homogeneous Poisson process on $\R_+$ with intensity $\theta$, and $\epsilon_1,
\epsilon_2,\dots$ are i.i.d.\ random vectors on the unit sphere
$\S^{n-1}$. Thus their distribution is characterised by two
parameters: the Poisson process intensity $\theta$ and the
probability measure $\hat{\sigma}$ on the sphere -- the distribution
of $\epsilon_k$'s. 
By the marking theorem for Poisson processes (\seg~\cite[Th.~5.6]{LastPenrose18}),
$\sum^\infty_{k=1}\delta_{(\Gamma_k, \epsilon_k)}$ is a Poisson on
$\R_+\times \S^{n-1}$ with intensity measure $\theta dt\times \hat{\sigma}$.
Hence we can appeal to the mapping theorem  (\seg~\cite[Th.~5.1]{LastPenrose18})
to see that 
\begin{align*}
\eta_\theta:=\sum^\infty_{k=1}\delta_{\Gamma_k^{-1/\alpha}\epsilon_k}
\end{align*}
is a Poisson process on $\R^n\setminus\{0\}$ with intensity measure
\begin{equation}\label{eq:levy}
\Lambda_\theta:=\theta\int_{\S^{n-1}}\int^\infty_0
\I\{t^{-1/\alpha}u\in\cdot\}\,dt\,\hat{\sigma}(du) =\theta\Lambda_1. 
\end{equation}
The right-hand side of \eqref{eq:lp} can be written as
a point process integral, so that
\begin{equation}\label{eq:lpp}
\xi_\theta:= \int u\, \eta_\theta(du),
\end{equation}
is a stable random vector with the given parameters.
The integrals here in this section are taken over $\R^n\setminus\{0\}$ unless specified
otherwise. 
The \stas\ distribution is infinitely divisible with
\emph{L\'evy measure} $\Lambda_\theta$ given by \eqref{eq:levy}. The measure
$\sigma=\theta\hat{\sigma}$ on $\S^{n-1}$
is called the \emph{spectral measure} of $\xi_\theta$ (or of $\Lambda_\theta$).
The convergence of the integral \eqref{eq:lpp} for $\alpha<1$ or for all
$\alpha<2$ in the case of a symmetrical spectral measure is
guaranteed, for instance, by \cite[Lem.~12.13]{Kallenberg02}. However, 
  the spectral measure need not be symmetric in order for the LePage
  representation \eqref{eq:lp} to hold. For instance, when
  $\alpha\geq1$ it is sufficient that a non-symmetric spectral measure
  satisfies $\int s\,\sigma(ds)=0$, see \cite[Th.~2]{BJP:01}.

By definition of $\Lambda_\theta$, we have for each Borel set
$B\subset\R^n\setminus\{0\}$ and each $c>0$ that
\begin{displaymath}
  \Lambda_\theta(c B)=c^{-\alpha}\Lambda_\theta(B).
\end{displaymath}
From \eqref{eq:lpp} and the above scaling property we obtain that
$\xi_\theta\deq \theta^{1/\alpha}\xi_1$.



For a $B\subset \S^{n-1}$, introduce the closed positive conical hull
\begin{displaymath}
\cone(B)=\mathrm{closure}\,\Bigl\{\sum_{k=1}^m \lambda_k x_k:\ x_k \in B,\
\lambda_k\geq 0,\ m \in \N\Bigr\}.
\end{displaymath}
Let $S_\sigma$ be the support of the spectral measure $\sigma$. It
follows from \eqref{eq:lpp} that the distribution of $\xi_\theta$ is
supported by $\cone(S_\sigma)$: the distribution of the sum
$\sum_{k=1}^m \Gamma_k^{1/\alpha} \epsilon_k$ is the convolution of
distributions with densities. Thus it is non-degenerate if
$\cone(S_\sigma)$ has a positive $n$-volume. Non-degenerate stable
laws possess an infinitely differentiable density: the distribution of
the first term $\Gamma_1^{-1/\alpha}\epsilon_1$ in \eqref{eq:lp} has
density with integrable derivatives of all orders implying that its
convolution with $\sum_{k=2}^\infty \Gamma_k^{-1/\alpha}\epsilon_k$
can also be differentiated under the integral any number of times
leading to expressions for the density derivatives. See also
\cite[Sec.~2.3.4]{Press:72} for an alternative proof.

We are now ready to formulate the main result of this section.
\begin{theorem}\label{th:stab}
   Let $\xi_\theta$ be a \stas\ random vector with LePage
  representation~\eqref{eq:lp} corresponding to the spectral measure
  $\sigma=\theta\hat{\sigma}$ such that $K:=\cone (S_{\hat{\sigma}})$
  has a positive $n$-volume. Then 
  \begin{itemize}
  \item[\textsl{(i)}] The density $f_\theta$ of $\xi_\theta$ satisfies
    \begin{equation}
      \label{alphadens}
      n f_\theta(x)+\langle x, \nabla f_\theta(x)\rangle = \alpha \int
      [f_\theta(x)-f_\theta(x-z)]\, \Lambda_\theta(dz),\quad  x\in\Int(K),
    \end{equation}
    where $\langle \cdot\,,\cdot\rangle$ is the scalar product in $\R^n$.
  \item [\textsl{(ii)}] Let $f_{|\xi_\theta|}$ denote
    the p.d.f.\ of the radius vector $|\xi_\theta|$. Then for all $r>0$,
\begin{equation}\label{radvec}
   r f_{|\xi_\theta|}(r)=\alpha \int
  \bigl[\P\big(|\xi_\theta|\leq r\big)-\P\big(|\xi_\theta+z|\leq
  r\big)\bigr]\Lambda_\theta(dz).
\end{equation}
\end{itemize}
\end{theorem}

\begin{corollary}
   The c.d.f.\ $F_\theta$ and the p.d.f.\ $f_\theta$ of a positive
   \stas\ on $\R_+$ with $0<\alpha<1$ are related through
    \begin{align}
      \label{alphadens1}
      & f_\theta(x)+ xf'_\theta(x) = \alpha^2\theta \int_0^x
      [f_\theta(x)-f_\theta(x-z)]z^{-\alpha-1}\,dz;\\
      \label{dimone}
  & x f_\theta(x)=\theta\alpha^2 \int_0^x
  \big[F_\theta(x)-F_\theta(x-z)\big]z^{-\alpha-1}\,dz \quad
  \text{for all $x>0$},
    \end{align}
\end{corollary}
\begin{proof}
  For a measurable $B\subset\R^n$ and a counting measure $\phi$,
  consider the indicator function
  \begin{displaymath}
    g_B(\phi)=\I\Bigl\{\int z\,\phi(dz)\in B\Bigr\}.
  \end{displaymath} 
  By \eqref{eq:lpp}, $\BE g_B(\eta_\theta)=\P(\xi_\theta\in B)$. Moreover,
  \begin{displaymath}
    \BE g_B(\eta_\theta+\delta_z)=\BE \I\Bigl\{\int u\,
    (\eta_\theta+\delta_z)(du) \in B\Bigr\}=\P(\xi_\theta + z\in B).
  \end{displaymath}
  Using \eqref{russopoiss} and noting that
$\Lambda_{\theta}=\theta\Lambda_{1}$,
we  obtain that for any measurable $B\subset\R^n$,
  \begin{align}
    \frac{d}{d\theta}\P(\xi_\theta\in B) & 
=\int\bigl[\P(\xi_\theta+z\in B)-\P(\xi_\theta\in B)\bigr]\,\Lambda_1(dz) \label{stabderiv1}\\
& = \frac{1}{\theta}\int[\P(\xi_\theta\in B-z)-\P(\xi_\theta\in B)]\,\Lambda_\theta(dz) \label{stabderiv2}
  \end{align}
Since $\xi_\theta\deq \theta^{1/\alpha}\xi_1$,
the density and its gradient satisfy
  \begin{align*}
    f_\theta(x)& =\theta^{-n/\alpha}f_1(\theta^{-1/\alpha}x),\\
    \nabla f_\theta(x)& = \theta^{-(n+1)/\alpha} \nabla
                        f_1(\theta^{-1/\alpha}x).
  \end{align*}
Therefore,
  \begin{align*}
    \frac{d}{d\theta} f_\theta(x) & 
   = -\frac{n}{\alpha}\theta^{-n/\alpha-1}f_1(\theta^{-1/\alpha}x) -
   \frac{1}{\alpha}\theta^{-n/\alpha}
   \langle \theta^{-1/\alpha-1}x,
   \nabla
   f_1(\theta^{-1/\alpha}x)\rangle\\
 & = -\frac{n}{\alpha\theta} f_\theta(x) - \frac{1}{\alpha\theta}
   \langle x, \nabla f_\theta(x)\rangle,
  \end{align*}

Take a set $B$ such that its closure is in $\Int(K)$. The density
  $f_\theta$ is bounded and the left-hand-side of
  \eqref{stabderiv1} becomes
  \begin{equation}\label{lhs}
    \frac{d}{d\theta}\P(\xi_\theta\in B)
    = \frac{d}{d\theta} \int_B  f_\theta(x)\,dx 
    = - \frac{1}{\alpha\theta}\int_B \big[ n f_\theta(x)
    +\langle x, \nabla f_\theta(x)\rangle\big]\,dx
  \end{equation}
  The right-hand-side of \eqref{stabderiv2} is
  \begin{displaymath}
    \frac{1}{\theta} \int \int_B [f_\theta(x-z)-f_\theta(x)]\,dx\,\Lambda_\theta(dz).
  \end{displaymath}
  Equating it to~\eqref{lhs}, we get the identity which holds for all
  measurable $B\subset \Int(K)$ which implies the
  identity~\eqref{alphadens} for almost all $x\in\Int(K)$. But the
  density is continuously differentiable there, so it also holds for
  all $x\in\Int(K)$.

  Recall that all one-dimensional \stas\ laws with $0<\alpha<1$ are
  totally skewed concentrated on either $\R_+$ or $\R_-$. Consider,
  for definitivness, a positive $\xi_\theta$. The spectral measure
  $\sigma$ is then $\theta\delta_1$ and \eqref{alphadens}
  becomes~\eqref{alphadens1}.

  Now let $B$ in~\eqref{stabderiv2} be the ball $B_r$  of radius $r$
  centred at the origin. Since 
  \begin{align*}
    \frac{d}{d\theta}\P(\xi_\theta\in B_r) & =\frac{d}{d\theta}
    \P(|\xi_\theta|\leq r) = \frac{d}{d\theta}
    \P(|\xi_1|\leq \theta^{-1/\alpha} r)\\
  & =
    -\frac{r}{\alpha}\theta^{-1-1/\alpha}
    \left.\frac{d}{dt}\P(|\xi_1|\leq t)\right|_{t=\theta^{-1/\alpha}r}
    =
    -\frac{r}{\alpha}\theta^{-1-1/\alpha}
    f_{|\xi_1|}(\theta^{-1/\alpha}r)
  \end{align*}
  and also
  \begin{align*}
    f_{|\xi_\theta|}(r) & =\frac{d}{dr} \P(|\xi_\theta|\leq r) =
                          \frac{d}{dr} \P(|\xi_1|\leq \theta^{-1/\alpha} r) \\
                        & = \theta^{-1/\alpha} 
                          \left.\frac{d}{dt} \P(|\xi_1|\leq t\})\right|_{t=\theta^{-1/\alpha}r}
=\theta^{-1/\alpha} f_{|\xi_1|}(\theta^{-1/\alpha}r),
  \end{align*}
  the relation \eqref{stabderiv2} takes the form~\eqref{radvec}.

  Notice that its one-dimensional variant \eqref{dimone}, when
  differentiated, gives~\eqref{alphadens1}.
\end{proof}

\section{Crofton's derivative formula for Poisson processes}
\label{secCroftonPoiss}

The classical Crofton formula \cite{Crof:1885} known in integral and
stochastic geometry relates the probability of events and, generally,
expectation of a random variable defined by configuration of a fixed
number of points uniformly distributed in a domain when the domain is
infinitesimally expanded. The property, described by the event or the
random variable should depend only on the mutual position of points,
so it must be rotation and translation invariant once all the points
are still in the domain, \seg~\cite[Ch.2]{KendMor:63}. We will revisit
this formula in Section~\ref{secBinomial}, but now we establish its
counterpart for Poisson processes.

Let $K\subset\R^n$ be a compact set and define
\begin{align}\label{parallel}
K_t:=K+tB^n=\{x+ty:\ x\in K,\ y\in B^n\},\quad t\ge 0,
\end{align}
where $B^n$ is the Euclidean unit ball.
This is the so-called {\em parallel set} of $K$ at distance $t$. 
Let $h\colon\R^n\to[0,\infty)$ be a continuous
function and let $\lambda$ be a measure
on $\R^n$ with Lebesgue density $h$.
For $t\ge 0$ let $\lambda_t$ be the restriction of $\lambda$ to
$K_t$ and let $\eta_t$ be a Poisson process
on $\R^n$ with intensity measure $\lambda_t$.
Let $g\colon\bN(\R^n)\to\R$ be measurable.
Under certain technical assumptions on $K$ and $g$, we shall prove that
\begin{equation}\label{CroftonPoisson}
  \left.\frac{d}{d t}\BE g(\eta_t)\right|_{t=0}=
  \int_{\partial K} \BE\big[g(\eta_0+\delta_x)-g(\eta_0)\big]h(x)\,\mathcal{H}^{n-1}(dx),
\end{equation}
where $\partial K$ is the boundary of $K$ and $\mathcal{H}^{n-1}$ is
the $(n-1)$-dimensional Hausdorff measure on $\R^n$.

Our main technical geometrical tool are the support measures
from \cite{HLW04}. We recall here briefly their definition and
main properties. We put $p(K,z):=y$ whenever $y$ is a uniquely determined point in $K$ with
$d(K,z):=\min\{x-z:x\in K\}=|y-z|$, and we call this point the {\em metric projection} of $z$ on $K$. 
If $0<d(K,z)<\infty$ and $p(K,z)$ is defined, then $p(K,z)$ lies on the boundary
$\partial K$ of $K$ and we put $u(K,z):=(z-p(K,z))/d(K,z)$. 
The {\em exoskeleton} $\exo(K)$ of $K$ consists of all points of $\R^n\setminus K$
which do not admit a metric projection on  $K$.
The {\em normal bundle} of $K$ is defined 
by
$$
N(K):=\{(p(K,z),u(K,z)):z\notin K\cup\exo(K)\}.
$$
It is a measurable subset of $\partial K\times \BS^{n-1}$,
where $\BS^{n-1}:=\{x\in\R^n:\|x\|=1\}$ is the unit sphere in $\R^n$.
The {\em reach function} 
$\delta(K,\cdot):\R^n\times \BS^{n-1}\rightarrow[0,\infty]$ of $K$ is
defined by
$$
\delta(K,x,u):=\inf\{t\ge 0:x+tu\in\exo(K)\},\quad (x,u)\in N(K),
$$
and $\delta(K,x,u):=0$ for $(x,u)\notin N(K)$.
Note that $\delta(K,\cdot)>0$ on $N(K)$.

We write $x\wedge y$ for $\min\{x,y\}$. By Theorem 2.1 in \cite{HLW04},
there exist signed measures $\mu_0(K;\cdot),\ldots,\mu_{n-1}(K;\cdot)$
on $\R^n\times\BS^{n-1}$ satisfying
\begin{align}\label{finite} 
\sum^{n-1}_{i=0}\int_{N(K)}(\delta(K,x,u)\wedge
  r)^{n-i}\,|\mu_i|(K;d(x,u))<\infty,\quad r>0, 
\end{align}
and, for each measurable bounded function $f:\R^n\rightarrow\R$ with
compact support, we have the following {\em local Steiner formula}: 
\begin{multline}\label{steiner}
\int_{\R^n\setminus K} f(x)\,dx
=\sum^{n-1}_{i=0}\omega_{n-i}
\int^\infty_0\int_{N(K)}s^{n-1-i}\I\{s<\delta(K,x,u)\}\\
\times f(x+su)\,\mu_i(K;d(x,u))\,ds,
\end{multline}
where $\omega_j:=j\kappa_j$ and $\kappa_j$ is the volume of the unit ball in $\R^j$.
These measures are  called {\em support measures} of $K$. They are uniquely
defined by~\eqref{steiner}
and the requirement $|\mu_i|(K;\R^n\times\BS^{n-1}\setminus N(K))=0$.
In general, the total variation measures $|\mu_i|(K;\cdot)$ featuring in \eqref{finite}
are not finite. However, it follows from \eqref{steiner} that
\begin{align}\label{finite2} 
\sum^{n-1}_{i=0}\int_{N(K)}\I\{\delta(K,x,u)\ge r\}\,|\mu_i|(K;d(x,u))<\infty,\quad r>0.
\end{align}
Therefore the integrals on the right-hand side of \eqref{steiner} are well-defined.
An important special case is that of a convex set $K$.
Then $\delta(K,x,u)=\infty$ for all $(x,u)\in N(K)$.

We start with the following proposition of independent interest.
For $i\in\{1,2\}$ we define $\partial^iK$ as the set of
all $x\in\partial K$ such that $\card\{u\in\BS^{n-1}:(x,u)\in N(K)\}=i$.

\begin{proposition}\label{p6.1} Let $t_0>0$ and let
$f\colon \R^n\to\R$ be continuous on $K_{t_0}$.
Then
the  right and left derivatives of  $t\mapsto \int_{K_t\setminus K}f(x)\,dx$
exist on $(0,t_0)$ and are given by
\begin{align}\label{e6.4}
\frac{d^+}{d t}\int_{K_t\setminus K}f(x)\,dx&=\int_{\partial^1 K_t} f(x)\,\mathcal{H}^{n-1}(dx),\\\notag
\label{e6.5}
\frac{d^-}{d t}\int_{K_t\setminus K}f(x)\,dx&=\int_{\partial^1 K_t} f(x)\,\mathcal{H}^{n-1}(dx)\\
&\quad+\sum^{n-1}_{i=0}\omega_{n-i}\int\I\{t=\delta(K;x,u)\}t^{n-1-i}f(x+tu)\,\mu_i(K;d(x,u)).
\end{align}
Moreover, if
\begin{align}\label{e6.6}
\sum^{n-1}_{i=0} \int (\delta(K;x,u)\wedge 1)^{n-i-1}\,|\mu_i|(K;d(x,u))<\infty,
\end{align}
then 
\begin{align}\label{e6.11}
\left.\frac{d}{d t}\int_{K_t\setminus K}f(x)\,dx\right|_{t=0}
=\int_{\partial^1 K} f(x)\,\mathcal{H}^{n-1}(dx)+2\int_{\partial^2 K} f(x)\,\mathcal{H}^{n-1}(dx).
\end{align}
\end{proposition}
\begin{proof} Let $t\in(0,t_0)$ and $r>0$ such that $t+r\le t_0$.
By the Steiner formula \eqref{steiner} 
\begin{align*}
\frac{1}{r}\int_{K_{t+r}\setminus K_t} f(x)\,dx
=\sum^{n-1}_{i=0}\omega_{n-i}
\int_{N(K)}r^{-1}\int^{t+r}_ts^{n-1-i}\I\{s<\delta(K,x,u)\}\\
\times f(x+su)\,ds\,\mu_i(K;d(x,u)).
\end{align*}
Since $f$ is continuous on $K_{t_0}$, there exists $c\ge 0$ such that
$|f(x+su)|\le c$ for all $(x,u)\in N(K)$ and $s\le t_0$. Moreover
we have for each $i\in\{0,\ldots,n-1\}$ that 
\begin{align*}
(n-i)r^{-1}\int^{t+r}_ts^{n-1-i}\I\{s<\delta(K,x,u)\}\,ds
&\le \I\{t<\delta(K,x,u)\}r^{-1}((t+r)^{n-i}-t^{n-i})\\
&\le c_i\I\{t<\delta(K,x,u)\}
\end{align*}
for some $c_i\ge 0$ (depending on $t$ but not on $r$). By \eqref{finite2} and continuity of $f$
we can apply the dominated convergence theorem to conclude that
\begin{align*}
\lim_{r\to 0+}\frac{1}{r}\int_{K_{t+r}\setminus K_t} f(x)\,dx
=\sum^{n-1}_{i=0}\omega_{n-i}\int_{N(K)}\I\{t<\delta(K,x,u)\}t^{n-1-i}f(x+tu)\,\mu_i(K;d(x,u)).
\end{align*}
By Corollary 4.4 in \cite{HLW04} the above right-hand side equals (note that $\omega_1=2$)
\begin{align*}
2\int_{N(K_t)} f(x)\,\mu_{n-1}(K_t;d(x,u)).
\end{align*}
By Proposition 4.1 in \cite{HLW04} we have for any compact set $A\subset\R^n$,
that 
\begin{align}\label{e6.23}
2\mu_{n-1}(A;\cdot)=\int_{\partial^1 A} \I\{x\in\cdot\}\,\mathcal{H}^{n-1}(dx)
+2\int_{\partial^2 A} \I\{x\in\cdot\}\,\mathcal{H}^{n-1}(dx).
\end{align}
Since $\partial^2 K_t=\emptyset$ (recall that $t>0$) we obtain the first assertion \eqref{e6.4}.

Similarly we obtain for the left derivative
\begin{align*}
\lim_{r\to 0+}\frac{1}{r}\int_{K_{t}\setminus K_{t-r}} f(x)\,dx
=\sum^{n-1}_{i=0}\omega_{n-i}\int_{N(K)}\I\{t\le \delta(K,x,u)\}t^{n-1-i}f(x+tu)\,\mu_i(K;d(x,u)).
\end{align*}
Writing $\I\{t\le \delta(K,x,u)\}=\I\{t= \delta(K,x,u)\}+\I\{t< \delta(K,x,u)\}$,
we can prove \eqref{e6.5} as before.

Assuming \eqref{e6.6}, the proof of \eqref{e6.11} again
follows from the Steiner formula, dominated convergence
and \eqref{e6.23}. Details are left to the reader.
\end{proof}

Let us define $I_K$ as the set of all $t>0$ such that
\begin{align}\label{exceptional}
\sum^{n-1}_{i=0}\int\I\{t=\delta(K;x,u)\}\,|\mu_i|(K;d(x,u))=0.
\end{align}
In view of \eqref{finite2} the set $(0,\infty)\setminus I_K$ is at most countably infinite.

\begin{theorem}\label{t6.2} Let $g\colon\bN\rightarrow\R$ be measurable
and $t_0>0$ such that $\BE |g(\eta_{t_0})|<\infty$ and
$x\mapsto \BE g(\eta_t+\delta_x)$ is continuous on $K_{t_0}$ for each $t<t_0$. 
Assume also that there exists $c>0$ 
such that
\begin{align}\label{e6.1}
\big|\BE D^k_{x_1,\dots,x_k}g(\eta_t)\big|\le c^k,\quad x_1,\dots,x_k\in K_{t_0},\,t\le t_0,\,k\in\N.
\end{align}
Then $t\mapsto \BE g(\eta_t)$ is differentiable on $I_K\cap (0,t_0)$ and the derivative is given by
\begin{equation}\label{21}
\frac{d}{d t}\BE g(\eta_t)=
\int_{\partial^1 K_t} \BE\big[g(\eta_t+\delta_x)-g(\eta_t)\big]
h(x)\,\mathcal{H}^{n-1}(dx). 
\end{equation}
Moreover, if \eqref{e6.6} holds, then
\begin{align}\label{e6.17}
\left.\frac{d}{d t}\BE g(\eta_t)\right|_{t=0}
=\sum^2_{j=1}j\int_{\partial^j K}
  \BE\big[g(\eta_0+\delta_x)-g(\eta_0)\big]
  h(x)\,\mathcal{H}^{n-1}(dx). 
\end{align}
\end{theorem}
\begin{proof}
Let $t\in[0,t_0)$ and let $r>0$ be such that $t+r\le t_0$.  The intensity
measure of the process $\eta_{t+r}$ equals the sum of $\lambda_t$ and
the restriction $\nu_r$ of $\lambda$ to $K_{t+r}\setminus K_t$. Thus,
by \eqref{emar4} (for $\nu=\nu_r$ and $\theta=1$)
\begin{align*}
\BE g(\eta_{t+r})=\BE g(\eta_t)+\int_{K_{t+r}\setminus K_t} \BE D_xg(\eta_t)h(x)\,dx+R(t,r),
\end{align*}
where
\begin{align*}
  R(t,r):=\sum^\infty_{k=2}\frac{1}{k!}\int_{(K_{t+r}\setminus K_t)^k}
  \BE D^k_{x_1,\ldots,x_k}g(\eta_t)\, h(x_1)\cdots h(x_k)\,d(x_1,\ldots,x_k). 
\end{align*}
We have that
\begin{align*}
|R(t,r)|\le \sum^\infty_{k=2}\frac{c^k}{k!}\int_{(K_{t+r}\setminus K_t)^k}
h(x_1)\cdots h(x_k)\,d(x_1,\ldots,x_k)=\exp(c(t,r))-c(t,r)-1,
\end{align*}
where $c(t,r):=c\int_{K_{t+h}\setminus K_t}h(x)\,dx$.
If $t>0$ then Proposition \ref{p6.1} shows the convergence
$\lim_{r\to 0+}r^{-1}c(t,r)=c(t)$ for some $c(t)\in\R$.
Therefore
\begin{align*}
\limsup_{r\to 0+}r^{-1}|R(t,r)|
\le \lim_{r\to 0+}\frac{c(t,r)^2}{r} \frac{\exp(c(t,r))-c(t,r)-1}{c(t,r)^2}=0.
\end{align*}
Under assumption \eqref{e6.6}, we have \eqref{e6.11}
so that the above remains true for $t=0$.

Again by Proposition \ref{p6.1} we have that
\begin{align*}
\lim_{r\to 0+}r^{-1}\int_{K_{t+r}\setminus K_t} \BE D_xg(\eta_t)h(x)\,dx
=\sum^2_{j=1}j\int_{\partial^j K}
  \BE\big[g(\eta_t+\delta_x)-g(\eta_t)\big]
  h(x)\,\mathcal{H}^{n-1}(dx), 
\end{align*}
first for $t>0$ (then the second term can be skipped) and then under
the assumption \eqref{e6.6}, also for $t=0$.

Let us now assume that $t-r>0$. Then it follows as above that
\begin{align*}
-\lim_{r\to 0+}r^{-1}(\BE g(\eta_{t-r})-\BE g(\eta_t))
=\lim_{r\to 0+}r^{-1}\int_{K_{t}\setminus K_{t-r}} \BE D_xg(\eta_t)h(x)\,dx,
\end{align*}
so that Proposition \ref{p6.1} shows that
\begin{align}\label{e6.32}
\frac{d^-}{d t}\BE g(\eta_t)&=\int_{\partial^1 K_t} \BE D_xg(\eta_t)h(x)\,\mathcal{H}^{n-1}(dx)\\ \notag
&\quad+\sum^{n-1}_{i=0}\omega_{n-i}\int\I\{t=\delta(K;x,u)\}t^{n-1-i}
\BE D_{x+tu}g(\eta_t)h(x+tu)\,\mu_i(K;d(x,u)).
\end{align}
Choosing now $t\in I_K$, concludes the proof.
\end{proof}

A bounded function $g$ satisfies the integrability assumptions of 
Theorem~\ref{t6.2} for all $t_0>0$, so that \eqref{21} holds
under a rather weak continuity assumption
for each compact $K$. Equation \eqref{e6.17} requires 
\eqref{e6.6}, constituting a non-trivial assumption on $K$.
This assumption is certainly satisfied if $|\mu_i|(K;\R^n\times\BS^{n-1})<\infty$
for each $i\in\{0,\ldots,n\}$. This is the case, for instance, if $K$
has a {\em positive reach} or is a finite union of convex sets,
see \cite{HLW04}.

\section{Crofton's derivative formula for binomial processes}\label{secBinomial}

For $t\ge 0$ we let $K_t$ and $\lambda_t$ be as defined in the
beginning of Section \ref{secCroftonPoiss}.
We assume that $\lambda(K)>0$.
In this section we consider a {\em binomial process} $\xi_t^{(m)}$ 
of size $m\in\N$ with sample distribution $\lambda_t/\lambda_t(K_t)$.
This is a point process of the form
\begin{align*}
\xi_t^{(m)}=\delta_{X_1}+\cdots +\delta_{X_m},
\end{align*}
where $X_1,\ldots,X_m$ are independent random vectors in $\R^n$
with distribution $\lambda_t/\lambda_t(K_t)$.
It is convenient to let $\xi_t^{(0)}:=0$ be the null measure (a point
process with no point.) 
Let $g\colon\bN\rightarrow\R$ be a measurable and bounded function.
Under certain assumptions on $K$ and $g$, we wish to prove that
\begin{align}\label{Croftonbin}
\frac{d}{dt}\BE g(\xi^{(m)}_t)=\frac{m}{\lambda(K_t)}
\int_{\partial K_t} \BE \big[
  g(\xi^{(m-1)}_t+\delta_x)-g(\xi^{(m)}_t)\big]h(x)\,\mathcal{H}^{n-1}(dx). 
\end{align}
The heuristic and historic background of this formula
is explained in \cite{solomon78}. 
If the boundaries of the sets $K_t$ are smooth, then 
\eqref{Croftonbin} follows from more general results 
in \cite{Badd77}. Our proof is very different and relies on the Poisson
version from Section~\ref{secCroftonPoiss}.


Recall the definition of the set $I_K$ at \eqref{exceptional}.

\begin{theorem}\label{t7.2} Let $g\colon\bN\rightarrow\R$ be measurable and bounded
and let $m\in\N$ and $t_0>0$. Suppose
that $x\mapsto \BE g(\xi^{(m-1)}_t+\delta_x)$ is continuous on $K_{t_0}$ for each $t<t_0$. 
Then $t\mapsto \BE g(\xi^{(m)}_t)$ is differentiable on $I_K\cap (0,t_0)$ and the derivative is given by
\begin{align}\label{e7.3}
\frac{d}{dt}\BE g(\xi^{(m)}_t)=\frac{m}{\lambda(K_t)}
\int_{\partial^1 K_t} \BE\big[g(\xi^{(m-1)}_t+\delta_x)-
  g(\xi^{(m)}_t)\big]h(x)\,\mathcal{H}^{n-1}(dx). 
\end{align}
Moreover, if \eqref{e6.6} holds, then
\begin{align}\label{e6.19}
\left.\frac{d}{d t}\BE g(\xi^{(m)}_t)\right|_{t=0}
=\sum^2_{j=1}j\int_{\partial^j K} \BE\big[g(\xi^{(m-1)}_0+\delta_x)-
  g(\xi^{(m)}_0)\big]h(x)\,\mathcal{H}^{n-1}(dx). 
\end{align}
\end{theorem}
\begin{proof}
We are using the Poisson process $\eta_t$ introduced in the beginning
of the previous section, and the well-known distributional identity 
(see e.g.\ \cite[Proposition 3.8]{LastPenrose18})
\begin{align*}
\BP(\xi^{(m)}_t\in\cdot\,)=\BP(\eta_t\in\cdot\mid \eta_t(K_t)=m)
=h_m(\lambda(K_t))\BP (\eta_t(K_t)=m, \eta_t \in\cdot\,),
\end{align*}
where the function $h_m\colon[0,\infty)\rightarrow\R$ is defined by 
$h_m(u):=m! e^u u^{-m}$. Note that the derivative of $h_m$ 
is given by
\begin{align*}
h'_m(u)=h_m(u)-\frac{m}{u}h_m(u).
\end{align*}
Let $t\in I_K\cap (0,t_0)$. We apply  \eqref{21} to the function 
$\tilde{g}(\varphi):= \I\{\varphi(\R^n)=m\}g(\varphi)$. 
Since $\BE g(\xi^{(m)}_t)=h_m(\lambda_t(K_t))\BE \tilde{g}(\eta_t)$ 
and $\tilde g(\eta_t+\delta_x)=\I\{\eta_t(\R^n)=m-1\}g(\eta_t+\delta_x)$, $x\in\R^n$,
this gives us
\begin{align*}
\frac{d}{d t}\BE g(\xi^{(m)}_t)=&
\Big[\frac{d}{d t} h_m(\lambda_t(K_t))\Big]\BE\I\{\eta_t(K_t)=m\}g(\eta_t) \\
&+ h_m(\lambda_t(K_t))\Big[\frac{d}{d t}\BE\I\{\eta_t(K_t)=m\}g(\eta_t)\Big].
\end{align*}
Taking into account \eqref{e6.4}, 
we obtain that the first summand equals
\begin{align*}
\BE g(\xi^{(m)}_t)\int_{\partial^1 K_t}h(x)\,\mathcal{H}^{n-1}(dx)
-\frac{m}{\lambda_t(K_t)}\BE g(\xi^{(m)}_t)\int_{\partial^1 K_t}h(x)\,\mathcal{H}^{n-1}(dx).
\end{align*}
By \eqref{21} the second summand equals
\begin{align*}
h_m&(\lambda_t(K_t))\int_{\partial^1 K_t}\BE \I\{\eta_t(K_t)=m-1\} g(\eta_t+\delta_x)h(x)\,\mathcal{H}^{n-1}(dx)\\
&-h_m(\lambda_t(K_t))\int_{\partial^1 K_t}\BE \I\{\eta_t(K_t)=m\} g(\eta_t)h(x)\,\mathcal{H}^{n-1}(dx)\\
=&\frac{m}{\lambda_t(K_t)}
\int_{\partial^1 K_t}\BE g(\xi^{(m-1)}_t+\delta_x)h(x)\,\mathcal{H}^{n-1}(dx)
-\BE g(\xi^{(m)}_t)\int_{\partial^1 K_t}h(x)\,\mathcal{H}^{n-1}(dx).
\end{align*}
Hence \eqref{e7.3} follows. The proof of \eqref{e6.19} is similar.
\end{proof}

\section*{Conclusion}
\label{sec:conclusion}

In this paper, we have shown how a perturbation method can both
clarify known results and generate new ones. Our approach is based on
applying infinitesimal changes to the parameters of the underlying
probabilistic model and using Margulis--Russo type perturbation
formulas to control the resulting variations.
It would be interesting to know whether our Theorems \ref{t6.2} and \ref{t7.2}
can be used to establish certain monotonicity properties of convex hulls,
such as those studied in \cite{Rademacher12,ReichenReitzner16}.


\end{document}